\documentclass[12pt]{amsart}     
\usepackage{amssymb}

\usepackage{amsmath}
\usepackage{amsthm}
\usepackage[T2A]{fontenc}
\usepackage[cp1251]{inputenc}

\newtheorem {thm}{Theorem}[section]

\newtheorem {cor}[thm]{Corollary}
\newtheorem {lem}[thm]{Lemma}
\theoremstyle{definition}
\newtheorem{df}{Definition}

\theoremstyle{remark}
\newtheorem{rem}{Remark}

\numberwithin{equation}{section}

\newcommand{\al}{\alpha}

\newcommand{\om}{\omega}
\newcommand{\omt}{\widetilde{\omega}}

\begin{document}





\title{Zygmund regularity of even  singular integral operators  on domains}


\author{Andrei V.\ Vasin}

\address{Admiral Makarov State University of Maritime and Inland Shipping,
Dwinskaya Street 5/7, St.~Petersburg, 198255, Russia}
\address{St.~Petersburg Department of V.A.~Steklov Mathematical Institute,
Fontanka 27, St.~Petersburg 191023, Russia}
\email{andrejvasin@gmail.com}

\begin{abstract}
Given a   bounded Lipschitz domain $D\subset \mathbb{R}^d,$   a convolution Calder\'{o}n-Zygmund operator $T$ and a growth function $\omega(x)$ of type $n$, we study  what conditions on the boundary of the domain are sufficient for boundedness of the restricted even operator $T_D$  on the  generalized Zygmund space  $C^{\om}_*(D)$.
 Based on a recent T(P) theorem, we prove that this holds if the smoothness of the boundary of a domain $D$ is    by one point, in a sense, greater than  the smoothness of the corresponding Zygmund space $C^{\om}_*(D)$. The main argument of the proof are the higher order gradient estimates of the transform  $T_D\chi_D$ of the characteristic function of a domain with the   polynomial boundary.
\end{abstract}

%
%

\maketitle



\section{Introduction}

\subsection{Backround}
A $C^{k}$ smooth  homogeneous Calder\'{o}n-Zygmund
 operator is a principal
value convolution operator
\[
Tf(y)= PV \int f(x) K(y-x)\, dx,
\]
 where $dx$ denotes  Lebesgue measure in $\mathbb{R}^d$ and
\[ K(x) =\frac{\Omega(x)}{|x|^d },\quad x \neq 0,\]
$\Omega(x)$ is a homogeneous function of degree 0 and  $\Omega(x)$ is $C^k$-differentiable  on
$\mathbb{R}^d \setminus \{0\}$  with zero integral on the unit sphere. The function $K(x)$ is called a Calder\'{o}n-Zygmund kernel.
Given a  domain $D \subset\mathbb{R}^d,$ we consider a modification of $T$. The  operator  defined by the formula
$$T_Df=  (Tf)\chi_D,\quad \mathrm{supp} f\subset \mathrm{clos}(D),$$
 is called  a \textit{restricted} Calder\'{o}n-Zygmund operator.

 The results of the present paper are  motivated by the
the next theorem of Mateu, Orobitg and Verdera  \cite{MOV} (see also Anikonov \cite{An}), who studied Lipschitz regularity
of quasiconformal mappings.
 \begin{thm}[{\cite[Main Lemma]{MOV}}]\label{thm1}
   Let $D$ be a   bounded domain with the  $C^{1+\alpha}$-smooth boundary, $0<\al<1$. Then the restricted Calder\'{o}n-Zygmund operator $T_D$ with an even kernel maps the Lipschitz space $Lip_\alpha(D)$ into itself.
\end{thm}

Observe that    the restricted  Calder\'{o}n-Zygmund operators  are not bounded in the Lipschitz spaces $Lip_\alpha(D)$ for  domains of  general kind. A computation shows this for the Hilbert transform on an interval  or for the Beurling transform on a square (see \cite{MOV}).
Also, one can see that Theorem \ref{thm1} holds when the  $C^{1+\alpha}$ smoothness of the boundary of a domain $D$ is    exactly by one point greater than  the smoothness of the corresponding Lipschitz space $Lip_\alpha(D)$.

  Theorem \ref{thm1} is extended in \cite{V1} to certain  spaces  of zero smoothness between  $Lip_\alpha(D)$ and $BMO(D)$.    Here  we extend Theorem \ref{thm1} for  the higher smoothness  general Zygmund  spaces.

\subsection {Growth functions and smooth domains}

To define  smoothness we 
follow Jansson (see \cite{Ja2}), and   consider general growth functions.
  \begin{df}
 \label{df1}
  A continuous   increasing  function $\omega:[0,\infty)\rightarrow[0,\infty),\, \omega(0)=0$
   is called a  growth function of type $n$, if $n$ is the    positive integer   such that   the  following two  regularity properties are satisfied:
\begin{enumerate}
  \item  For some $q$, $ n<q<n+1$ the function $\frac{\om(t)}{t^q}$ is \textit{almost decreasing}, that is there exists a positive constant $C=C(q)$ such that
  \begin{equation}\label{eq:eq42}
 \omega(st)<Cs^q\omega(t),\;s>1.
\end{equation}
  \item For each  $r$, $r<n$, the function $\frac{\om(t)}{t^r}$ is \textit{almost increasing}, that is there exists a positive constant $C=C(r)$ such that
     \begin{equation}\label{eq:eq32}
 \omega(st)<Cs^{r}\omega(t), \;s<1.
\end{equation}
\end{enumerate}
            \end{df}

\begin{df}
 \label{df3}
 A bounded domain $D\subset  \mathbb{R}^d$ is called  a $(\delta, R)$-Lipschitz domain if, for  each  point $a\in\partial D$, there exists a function $A:\mathbb{R}^{d-1}\rightarrow \mathbb{R}$ with $\|\nabla A\|_\infty\leq \delta$  and a cube $Q$ with side length $R$ and centre $a$ such that, after a suitable shift and rotation, that sends $a$ to  origin,  one has
 $$D\cap Q=\{ (x',x_d)\in (\mathbb{R}^{d-1}, \mathbb{R})\cap Q:x_d>A(x')\}.$$
The cube $Q$ is called an $R$-window of the domain. Without risk of confusion, we omit parameters $\delta$ and $R$, and consider  Lipschitz domains in general.

 Given  a   growth function   $\omega$ of type $n$,      a bounded Lipschitz domain
$D$ is called a $C^{\om}$  domain if    the function $A$ from the definition satisfies the following condition:
there exists a polynomial $P(x')$ of order $n$ such that
     \[ |A(x')-P(x')|\lesssim_D \omega(|x'|), \:x=(x',x_d)\in Q .\]
\end{df}
\begin{rem}
     For a growth function  such that $\om(t)= o(t^n)$ the parametrization $A$ is  of class $C^{n}$ and $P$ is the corresponding Taylor polynomial. In general case of a  type $n$ growth function (e.g. $\om(t)=t^n$), the parametrization  is of class $C^{n-1}$ only.
\end{rem}

\subsection {Regularity of the transform of the characteristic function}

In our first theorem we relate the smoothness of the boundary $\partial D$ with the gradient estimates of  $T_D\chi_D$ near the boundary.
\begin{thm}
\label{thm5}
    Let $\om$ be a growth function of  type $n>1$, let $D\subset \mathbb{R}^d$ be a   $C^{\om}$ domain and let $T$ be  a $C^{2n-1}$-smooth  homogeneous Calder\'{o}n-Zygmund operator with an even kernel.  Then
  \begin{equation}
 \label{eq:eq222}
  | \nabla ^{n}T_D\chi_D(x)|\lesssim_{D, K} \frac{\om(\rho(x))}{\rho(x)^{n+1}},
    \end{equation}
    where $\rho(x)$ is distance from $x\in D$ to the boundary $\partial D$.
\end{thm}
For   $n=1$,  the result is obtained in \cite[Proposition 1.10]{V1}.
The  known argument \cite{An, MOV} to estimate $T_D\chi_D$ is to apply  an extra cancellation property of  even convolution Calder\'{o}n-Zygmund operators. This extra cancellation property means particularly that a derivative of the transform  of the characteristic
function of a half-space is zero out of its boundary.
For higher order of smoothness, $n>1$, we  need
to approximate the boundary of the domain by polynomial graphs instead of
hyperplanes.  Since a derivative of the  transform of the characteristic function of
a domain bounded by a polynomial graph of degree greater than one is not
zero anymore in general, the proof is more complicated.  In Section 3 we   obtain the higher order gradient  estimates $|\nabla ^{n}T_D\chi_D(x)|$ for the domains bounded by  polynomial graphs.
To do this job we will find  a new form of  extra cancellation property in  Lemma \ref{lem16}, which is suitable for the  polynomial domains in $\mathbb{R}^d$.

The results proved are related not only to  cited \cite{MOV}, but also to  Crus and Tolsa \cite{CT, T}, where the authors research Sobolev regularity of  the Beurling
transform $B$ of the characteristic function $\chi_D$ of a Lipschitz domain $D \in \mathbb{R}^2$. They obtained that the outward unit normal $N$ to the boundary,  and the trace of  $B\chi_D$ belong to the same Besov space $B^{\al-1/p}_{p,p}(\partial D)$.
 The results on Sobolev regularity in \cite{CT, T} of  the Beurling
transform $B$ of the characteristic function $\chi_D$ were extended on higher orders of smoothness by Prats \cite{Pr}. The author used there   one complex variable approach.

Note that the dependence on the smoothness of the boundary for estimate (\ref{eq:eq222})   is sharp.  It is  proved in  \cite{V1}  in a  case of the Beurling transform on plane domains and the  Dini regular growth function for $n=1$. For the case of Sobolev regularity it is proved in  \cite{ T}
In the present paper we do not consider  the sharpness (\ref{eq:eq222}).   Our approach to obtain  the gradient estimates  (\ref{eq:eq222})   assumes    $C^{2n-1}$ smoothness of  the kernel $K$. In any case, one may consider  the
 restricted Riecz multipliers as the  operator $T_D$.

\subsection {Zygmund spaces}

Let $dx$ denote Lebesgue measure in $\mathbb{R}^d.$ Let $Q$ be a cube in $\mathbb{R}^d$ with edges parallel to coordinate axes, let $|Q|$ denote the volume of $Q$ and let $\ell=\ell(Q)$ be  its side length. Also, let $\mathcal{P}_n$ be the space of polynomials of degree at most  $n$.
\begin{df}
 \label{df2}
  Given a growth function $\omega$ of    type $n>0$, the homogeneous generalized  Zygmund space $C_*^{\om}(D)$  in a domain $D\subset  \mathbb{R}^d$ consists of those $f\in L^1_{loc}(D,dx)$ for which the  seminorm
\begin{equation}\label{eq:eq1}
   \|f\|_{\omega,D}=\sup_{Q\subset D} \inf_{P\in\mathcal{P}_n}\frac{1}{\omega(\ell)} \|f-P\|_{L^{\infty} (Q, dx/|Q|)}.
\end{equation}
is finite.
\end{df}
\begin{rem}
It clearly holds $C^{\om}(D)\subseteq C_*^{\om}(D)$. Particularly,
  for a  type $n\geq 1$  growth function $\om(t)=t^{s}$ with real $s$, $n<s<n+1$, both the spaces   coincide, and are called  the H\"older space $C^{s}(D)$ of order $s$.
On the other hand,
   for  $\om(t)=t^{n}$ with positive integer $n$,  one has the strict embedding, where $C_*^{\om}(D)$ is the classical  Zygmund space  $C^{n}_*(D)$, while   $C^{\om}(D)$ is the Lipschitz of order $n$ space  $C^{n-1,1}(D)$.
\end{rem}
Recently ( \cite{VD}),  it was proved   a criteria  of the boundedness of the  Calder\'{o}n-Zygmund operator  on   the Zygmund space.
To formulate the result,  define  the  associated  growth function as following
 \begin{equation}
 \label{eq:eq2}
   \widetilde{\omega}(x)= \frac{\omega(x)}{\max \{1,\int_x^1 \omega(t)t^{-n-1}dt\}}.
    \end{equation}

\begin{thm}[{\cite[Theorem 1.3]{VD}}]\label{thm3}
 Let $\omega$ be a   growth function  of type $n\in \mathbb{N}$  and let $D\subset \mathbb{R}^d$ be a  bounded Lipschitz domain. Let $T$ be  a $C^{n+1}$-smooth  homogeneous Calder\'{o}n-Zygmund operator. Then the  restricted operator $T_D$ is bounded on the  space $C_*^{\om} (D)$ if and only if two following properties have place:
\begin{enumerate}
  \item  $T_DP \in C_*^{\om} (D)$ for any polynomial $P\in\mathcal{P}_n(D)$.
  \item For any cube $Q \in D$ with the centre $x_0$ and  for any  homogeneous  polynomial $P_{x_0}(x)= P(x-x_0)$ of degree $n$, there exists a polynomial  $S_Q\in\mathcal{P}_n(D)$ such that
     \[\|T_DP_{x_0}-S_Q\|_{L^1(Q, dx/|Q|)}\leq C \|P\| \omt(\ell) \]
  with a constant $C$ independent of  $Q$.
\end{enumerate}
\end{thm}
\begin{rem} These type of arguments were refered by   Prats and Tolsa \cite {PT} as T(P) theorem
to indicate explicitly that the corresponding characterization uses values
of the operator $T$ on the polynomials of appropriate degree.
\end{rem}
\begin{rem}If a  growth function $\omega$ of type $n$ is Dini  regular,  that is,   the integral
 \[\int_0 \omega(t)t^{-n-1}dt\]
 converges, then   $\widetilde{\omega}(x)$ is equivalent to $\omega(x)$. In this case the second condition in the theorem  follows from the first one and may be omitted.
 \end{rem}
Checking the conditions  of  Theorem \ref{thm3}, we obtain the main result, where  we extend Theorem \ref{thm1} to the higher smoothness Zygmund spaces. Given a type $n$  growth function $\om(t)$, we denote by   $C^{\om,1}$ domain $D$ for the type $n+1$ growth function $t \om(t)$.
\begin{thm}
\label{thm4}
Let $\om$ be a growth function of   type $n$, $n\geq 1$,  and let  $D\subset \mathbb{R}^d$ be a  bounded $C^{\om,1}$ domain.  Let $T$ be  a  $C^{2n+1}$-smooth  homogeneous Calder\'{o}n-Zygmund operator with an even kernel. Then  $T_D$ is bounded on the  space  $C_*^{\om} (D)$.
\end{thm}

  For the classical Zygmund and the H\"older spaces  Theorem \ref{thm4}   states the next smoothness drop by one point.
\begin{cor}
For a positive integer $n$,     let $T$ be  a  $C^{2n+1}$-smooth  homogeneous Calder\'{o}n-Zygmund operator with an even kernel.
Then
\begin{enumerate}
  \item   $T_D$ is bounded on the  H\"older space  $C^{s} (D)$, provided $D$ is $C^{s+1}$ domain, $n<s<n+1$;
  \item $T_D$ is bounded on the original Zygmund space $ C^{n}_*(D)$, if  the relevant  domain $D$ has  the   $C^{n,1}$  boundary.
\end{enumerate}
\end {cor}

\subsection {Organization and notation}

 In Section 2 we prove  auxiliary facts about differentiation of PV integrals concerning the extra cancellation property.
Theorems \ref{thm5} and \ref{thm4} are obtained in Section 3 and 4, respectively.

As usual, the letter $C$ will denote a constant, which may be different at each
occurrence and which is independent of the relevant variables under consideration.
Notation $A\lesssim B$ means that there is a fixed positive constant $C$ such that $A<CB.$ If  $A\lesssim B\lesssim A,$ then we write $A\approx B.$ We write $A\lesssim_{a,b} B,$ if a corresponding constant depends on $a,b.$

\section{Cancellation property and differentiation of PV integrals}
\subsection{Cancellation property of even kernels.}
The following proposition
 plays a crucial role in  the study of a  smooth convolution Calder\'{o}n-Zygmund operator with an even kernel.
\begin{lem}[{\cite{I,MOV}}]
\label{lem14}
  Let   $B$ be an arbitrary Euclidean ball in $\mathbb{R}^d,$  let $T$ be  a smooth  homogeneous Calder\'{o}n-Zygmund operator with an even kernel. Then $(T\chi_B)\chi_B\equiv 0$.
\end{lem}
In fact, the authors proved in \cite{MOV} that
\begin{equation}\label{eq:eq130}
 \int_{B\setminus \{x: |x-y|<\varepsilon\}}  K(x-y)dx= 0,
\end{equation}
where $y \in B$ and $\varepsilon< dist(y, \partial B)$.

In what follows, we  use a reformulation of  the cancellation  property of the even  functions with zero integral on the unit sphere.
\begin{lem}
\label{lem16}
  Let    $K$ be  an even  homogeneous function of degree $-d$ in  $\mathbb{R}^d \setminus \{0\}$ with zero integral on the unit sphere. Then
 \[\int_{\mathbb{R}^{d-1}}  K(x',1)dx'=0,\]
provided   $dx'$ is Lebesgue measure on $\mathbb{R}^{d-1}$.
\end{lem}
\begin{proof}
Put $x'=ur$, $u\in \mathbb{S}^{d-2}$, $r=|x'|$ and integrate in spherical coordinates in $\mathbb{R}^{d-1}$.  We have
 \[\int_{\mathbb{R}^{d-1}}  K(x',1)dx'=\int_{S^{d-2}}dS_{d-2}(u)  \int_{0}^{\infty} K(u r,1)r^{d-2}dr ,\]
where $dS_{d-2}(u)$ is the induced surface measure on the unit sphere $S^{d-2}$.
Using the new variable  $r=\tan \theta$, $\theta\in (0, \pi/2)$, in the inner integral,  we obtain  by homogeneity
\[=\int_{S^{d-2}}  dS_{d-2}(u)\int_{0}^{\pi/2} K(u \tan \theta,1)\tan^{d-2}\theta\cos^{-2}\theta d\theta\]
\[=\int_{S^{d-2}}  dS_{d-2}(u) \int_{0}^{\pi/2} K(u \sin \theta,\cos\theta)\sin^{d-2}\theta d\theta\]
\[=\int_{U}  K(x) dS_{d-1}(x),\]
where the set $U$ is the  half of the unit sphere $S^{d-1}$ above the hyperplane $x_d=0$ and $dS_{d-1}(x)$ is the induced surface measure on the sphere $S^{d-1}$. Since $K$ is even, the last integral is zero and the proof of the lemma is completed.
\end{proof}

From the Lemma \ref{lem16} one easily has
\begin{cor}
 Let    $K$ be  an even  homogeneous function of degree $-d$ in  $\mathbb{R}^d \setminus \{0\}$ with zero integral on the unit sphere. Let $ \mathbb{H}^{d-1}$
 be an affine hyperplane that does not pass through the origin, and denote  by   $dx$ Lebesgue measure in $ \mathbb{H}^{d-1}$. Then
 \[\int_{\mathbb{H}^{d-1}}  K(x)dx=0.\]
 \end{cor}

\subsection{Differentiation of $PV$  integrals with even kernels.}
We  prove the  differentiation formula for $PV$ integrals for even Calder\'{o}n-Zygmund kernels.
  \begin{lem}
\label{lem13}
  Let $ K(x)$
 be an even $C^1$ smooth Calder\'{o}n-Zygmund kernel in $\mathbb{R}^d$. Let $y$ be in a $C^{1}$ domain $D$,
 $D_\varepsilon=D\setminus \{x: |x-y|<\varepsilon\}$,   $\varepsilon< \textrm{dist}(y,\partial D)$. Then,
\begin{equation}\label{eq:eq137}
\frac{\partial}{\partial y_i} \int_{D_\varepsilon} K(x-y)dx=\int_{D_\varepsilon} \frac{\partial}{\partial y_i}K(x-y)dx,
\end{equation}
and tending $\varepsilon$ to $0$,
\begin{equation}\label{eq:eq13}
\frac{\partial}{\partial y_i} PV\int_D K(x-y)dx=PV\int_D \frac{\partial}{\partial y_i}K(x-y)dx.
\end{equation}

\end{lem}
\begin{proof}

We start with the  formula  of differentiation under the integral sign  (see \cite[Section 8(3)]{Mikh}). Notice that the restriction of having an even function is not there anymore.
\begin{equation}
\label{eq:eq14}
\frac{\partial}{\partial y_i} \int_{D_\varepsilon} K(x-y)dx=  \int_{D_\varepsilon} \frac{\partial}{\partial y_i} K(x-y)dx +\int_{S_\varepsilon}K(x-y) \cos(\nu,x_i)dS(x),
\end{equation}
where   $S_\varepsilon=\{x: |x-y|=\varepsilon\}$,   $\nu$ is  the outer normal to $S_\varepsilon$, and $dS(x)$ is the induced  surface measure on $S_\varepsilon$.

 Green's formula, applied to the first integral on the right hand side  of (\ref{eq:eq14}), provides
 \[\int_{D_\varepsilon} \frac{\partial}{\partial y_i} K(x-y)dx= -\int_{D_\varepsilon} \frac{\partial}{\partial x_i} K(x-y)dx\]
 \[= -\int_{\partial D \cup S_\varepsilon}K(x-y) \cos(\nu,x_i)dS(x).\]
Therefore,
\begin{equation}
\label{eq:eq141}
\frac{\partial}{\partial y_i} \int_{D_\varepsilon} K(x-y)dx=-\int_{\partial D}K(x-y) \cos(\nu,x_i)dS(x).
\end{equation}
 Replace the  domain $D$ in (\ref{eq:eq141}) by an arbitrary ball $B$. By  (\ref{eq:eq130}), for $y\in B$  we have
\[0= \frac{\partial}{\partial y_i} \int_{B_\varepsilon} K(x-y)dx=-\int_{\partial B}K(x-y) \cos(\nu,x_i)dS(x),\]
where  $B_\varepsilon=B\setminus \{x: |x-y|<\varepsilon\}$ and  $\partial B$ is a boundary of $B$.
Therefore, the second integral on the right hand side of (\ref{eq:eq14}) is equal to zero for any $\varepsilon$, and  (\ref{eq:eq137}) holds.

On the other hand, the surface integral on right hand side of (\ref{eq:eq141}) is independent of $\varepsilon$, so, by Differentiable Limit Theorem,
  we obtain (\ref{eq:eq13}).
 Lemma \ref{lem13} is proved.
\end{proof}

 Iterating the formula obtained and using (\ref{eq:eq130})  we have
\begin{cor} \label{cor17}
Let $ K(x)$
 be an even $C^n$ smooth Calder\'{o}n-Zygmund kernel in
$\mathbb{R}^d$, $D_\varepsilon=D\setminus \{x: |x-y|<\varepsilon\}$,   $\varepsilon< \textrm{dist}(y,\partial D)$ and let  $\partial^k_y$ be the derivative of order $k$, $k=1,\dots,n$. Then
\begin{equation}\label{eq:eq131}
\partial^k_y \int_{D_\varepsilon} K(x-y)dx=\int_{D_\varepsilon} \partial^k_y K(x-y)dx
\end{equation}
and respectively
\begin{equation}\label{eq:eq313}
\partial^k_y PV \int_{D} K(x-y)dx=PV \int_{D} \partial^k_y K(x-y)dx
\end{equation}
for each $y$ in a $C^{1}$ domain $D$.
\end{cor}

By Green's formula, one has

\begin{lem}
\label{lem17}
  Let $K$ be a $C^n$-smooth Calder\'{o}n-Zygmund kernel in
$\mathbb{R}^d$. Let $P_k(x)$ be a homogeneous polynomial of order $k$ and let $\partial^k_x$ be a partial derivative of order $k$. Then the functions
    $N_1(x)=\partial^k_x (P_kK)(x)$ and  $N_2(x)=P_k(x)\partial^k_x K(x)$, $x\in \mathbb{R}^d \setminus 0$,  are   $C^{n-k}$-smooth  Calder\'{o}n-Zygmund kernels.
   \end{lem}
\begin{proof}
It is easy to see that  $N_1$ and $N_2$ are  homogeneous of order $-d$ functions, $C^{n-k}$-smooth on $\mathbb{R}^d\backslash0$.
That a degeneration can occur, for example, for the generalized Riesz potential, doesn't bother us. We need to check  that
\[\int_{r<|x|<R}N_i(x)dx=0, \:i=1,2,\]
for arbitrary $0<r<R<\infty$.

We argue by  induction on order of polynomial $P_k(x)$. So,  we start with
the function $N_1=\partial_j(x_k K(x))$. Let $\nu$ be  the outer normal to the spheres $S_r$ and  $S_R$, and  let $dS(x)$ be the induced  surface measure. By Green's formula, we have
\[\int_{r<|x|<R}N_1(x)dx=\int_{r<|x|<R}\partial_j (x_k K(x))dx\]
\[=\int_{S_R\cup S_r}x_k K(x)\cos(\nu, x_j) dS(x)\]
\[=\int_{S_R}x_k K(x)\cos(\nu, x_j) dS(x)-\int_{S_r}x_k K(x)\cos(\nu, x_j) dS(x).\]
By homogeneity, the integral $\int_{S_R}x_k K(x)\cos(\nu, x_j) dS(x)$ does not depend on $R$,  therefore
$\int_{r<|x|<R}N_1(x)dx=0$.

 For  $N_2=x_k \partial_j K(x)$,  we have
\[\int_{r<|x|<R}N_2(x)dx=\int_{r<|x|<R}x_k \partial_j K(x)dx\]
\[=\int_{r<|x|<R}N_1(x)dx-\delta_{k,j}\int_{r<|x|<R} K(x)dx=0,\]
since both integrals are equal to zero.

For $N_1$ defined by an arbitrary polynomial $P_k$, we repeat the argument from the first step. So, by  Green's formula for $-d$-homogenous functions, we obtain
\[\int_{r<|x|<R}N_1(x)dx=\int_{r<|x|<R}\partial^k_x (P_k(x)K(x))dx\]
\[=\int_{r<|x|<R}\partial_j \partial^{k-1}_x (P_k K)(x)dx\]
\[=\int_{S_R\cup S_r}\partial^{k-1}_x (P_k K)(x) \cos(\nu, x_j) dS(x)=0\]
by  homogeneity.

For $N_2$ with an arbitrary polynomial $P_k$, we have
  \[N_2(x)=P_k(x)\partial^k_x K(x)=P_k(x)\partial_j\partial^{k-1}_x K(x)\]
\[=\partial_j(P_k \partial^{k-1}_x K)(x)-\partial_jP_k(x)\partial^{k-1}_x K(x).\]
Then
\[\int_{r<|x|<R}N_2(x)dx\]
\[=\int_{r<|x|<R}\partial_j(P_k \partial^{k-1}_x K)(x)dx-\int_{r<|x|<R}\partial_jP_k(x)\partial^{k-1}_x K(x)dx.\]
Apply  Green's formula to the first summand, then

\[\int_{r<|x|<R}\partial_j(P_k \partial^{k-1}_x K)(x)dx=\int_{S_R\cup S_r}P_k \partial^{k-1}_x  K(x) \cos(\nu, x_j) dS(x)=0\]
by homogenuity.
Since    $\textrm{deg}(\partial_jP_k)=k-1$, the induction assumption is applicable for the second summand.
 The proof of Lemma \ref{lem17} is completed.
\end{proof}

 By   Differentiable Limit Theorem, Lemma \ref{lem17} and Lemma \ref{lem13} imply a corollary.
\begin{cor}
Let $ K(x)$
 be an even   homogeneous of degree $-d$ function,  $C^n$-smooth  on
$\mathbb{R}^d \setminus 0$ with zero integral on the unit sphere. Let $P$ be a homogeneous polynomial of degree $k\geq 0$. Then, for the derivative of order $l$, $k<l\leq n$, we have
\begin{equation}\label{eq:eq132}
\partial^l_y PV\int_D K(x-y)P(x-y)dx=PV\int_D \partial^l_y (K(x-y)P(x-y))dx.
\end{equation}
\end{cor}
\begin{proof}
Replacing $K$ by $KP$ in  (\ref{eq:eq14}) and using Green's formula, we have

\[\partial_{y_i} \int_{D_\varepsilon} K(x-y)P(x-y)dx\]
\[=  \int_{D_\varepsilon} \partial_{ y_i} (K(x-y)P(x-y))dx +\int_{S_\varepsilon}K(x-y)P(x-y) \cos(\nu,x_i)dS(x)\]
\[= -\int_{\partial D}K(x-y)P(x-y) \cos(\nu,x_i)dS(x),\]
where  $D_\varepsilon=D\setminus \{x: |x-y|<\varepsilon\}$, $S_\varepsilon=\{x: |x-y|=\varepsilon\}$,  $\varepsilon< \textrm{dist}(y,\partial D)$,  $\nu$ is  the outer normal to $S_\varepsilon$, and $dS(x)$ is the induced  surface measure on $S_\varepsilon$.

Therefore, for any derivative $\partial^k_y$ of order $k$,
\[\partial^k_y \partial_{y_i}\int_{D_\varepsilon} K(x-y)P(x-y)dx=-\int_{\partial D} \partial^k_y (K(x-y)P(x-y)) \cos(\nu,x_i)dS(x).\]
By Lemma \ref{lem17}, $\partial^k_x (KP)(x)$ is  a   smooth  homogeneous Calder\'{o}n-Zygmund kernel; it is even, since $K$ is even.
Replace $K(x-y)$ by $\partial^k_y (K(x-y)P(x-y))$  in (\ref{eq:eq141}). We have
\[\partial_{y_i} \int_{D_\varepsilon}\partial^k_y (K(x-y)P(x-y))dx=-\int_{\partial D}\partial^k_y (K(x-y)P(x-y)) \cos(\nu,x_i)dS(x).\]
Comparing the last two formulas, one has
\[\partial^k_y\partial_{y_i} \int_{D_\varepsilon}K(x-y)P(x-y)dx= \partial_{y_i}\int_{D_\varepsilon}\partial^k_y (K(x-y)P(x-y))dx.\]
Iterating, we have
\[\partial^l_y \int_{D_\varepsilon}K(x-y)P(x-y)dx= \partial^{l-k}_{y}\int_{D_\varepsilon}\partial^k_y (K(x-y)P(x-y))dx.\]
with $k<l\leq n$.

By Corollary  \ref{cor17} applied to the even kernel $\partial^k_x (KP)(x)$, we obtain
\[\partial^l_y \int_{D_\varepsilon} K(x-y)P(x-y)dx=\int_{D_\varepsilon} \partial^l_y (K(x-y)P(x-y))dx,\]
with $k<l \leq n$. Finally, by   Differentiable Limit Theorem, we obtain (\ref{eq:eq132}).
\end{proof}

We finish  the section by the following lemma.
\begin{lem}
\label{lem15}
  Let   $\Pi$ be the  boundary   of a halfspace $\Pi_+$, let $ K(x)$
 be an even   homogeneous of degree $-d$ function,  $C^n$-smooth  on
$\mathbb{R}^d \setminus 0$ with zero integral on the unit sphere. Then for any derivative of order $0< k \leq n$ and any $y\in \Pi_+$, one has
 \[PV \int_{\Pi_+} \partial^k K(x-y)dx=0.\]
\end{lem}
\begin{proof}
After a shift and a rotation, we may assume that $\Pi$ is defined by the equation $x_d=0$, and $\Pi_+$ is the upper  halfspace defined by $x_d>0$, where $x=(x',x_d)\in \mathbb{R}^{d-1}\times\mathbb{R}$.
Consider  a sequence of imbedded balls $B_r$ with centres $(0',r)$ and radii $r>0$.
By (\ref{eq:eq130}) and (\ref{eq:eq131}), we have for $y\in B_r$
\[\int_{B_r \setminus \{x: |x-y|<\varepsilon\} } \partial^k K(x-y)dx=0.\]
Since  $\partial^k K(x-y)\in L^1(\mathbb{R}^d\setminus \{x: |x-y|<\varepsilon\},dx)$,  hence tending $r$ to $\infty$,
by   the Dominated Convergence  Theorem, it holds
\[\int_{B_r\setminus \{x: |x-y|<\varepsilon\}} \partial^k K(x-y)dx \longrightarrow \int_{\Pi_+\setminus \{x: |x-y|<\varepsilon\}} \partial^k K(x-y)dx\]
for arbitrary $y\in \Pi_+$ and $\varepsilon<dist(y, \Pi)$. That is  required, by Differentiable Limit Theorem.
\end{proof}

\section {Proof of Theorem \ref{thm5}}
 \subsection {$C^{\omega}$ parametrization.}  Fix a constant $r_0$ small enough so that  series of properties
that will be needed along the proof are satisfied.
   Choose $y\in D$ and put  $\delta=\textrm{dist}(y,\partial D)\leq r_0$. Assume that this distance is attained at a point $a \in\partial D$, which we can assume to be $a=0$.    We    assume   that $y=(0',\delta)$ and
 $(x',x_d)\in D$ for $0<t<r_0$.

 By Definition \ref{df3}, we have a function $A$, an $R$-window $Q$   with centre in origin  such that, after suitable shift and rotation
 \[D_Q=D\cap Q=\{ (x',x_d)\in (\mathbb{R}^{d-1}, \mathbb{R})\cap Q:x_d>A(x')\}.\]
 It is more convenient for us to take a  cylinder $Q=\{(x',x_d), |x'|<r_0, |x_d|<r_0\}$ as a $R$-window. Also, assume that  the tangent hyperplane to $\partial D$ in origin   is $x_d=0$,  $A(0')=0$ and $\nabla A(0')=0$.
For the relevant polynomial $P(x')\in \mathcal{P}_{n}(\mathbb{R}^{d-1})$  one has
   \[ |A(x')-P(x')|\lesssim_D \omega(|x'|), \:|x'|<r_0 .\]
 Observe that the polynomial $P$ inherits certain properties of the function $A$: $P(0')=0$ and $\nabla P(0')=0$.

\subsection {Start of proof.}
In order to prove Theorem \ref{thm3}   we will estimate for each derivative of order $n$ the following integral
$$I=\partial_y^{n} PV \int_D K(x-y)dx = PV\int_D  \partial_y^{n}K(x-y)dx, $$
where  equality is granted by Corollary \ref{cor17}.

Split   $I$ at the level $r_0$:
$$I= PV\int_{D_Q}\partial_x^{n}K(x-y)dx + \int_{D\setminus Q}\partial_x^{n}K(x-y)dx$$
$$ =II+III.$$
For  term  $III$, we clearly have
 $$ |III |\lesssim \frac{1}{r_0^{n}}. $$
To estimate term $II$,  consider the domain $W=\{x=(x',x_d): x_d>P(x')\}$ and denote $W_Q=W\cap Q$.  We represent $II$ as
$$II=\int_{D_Q \backslash W_Q}  \partial_x^{n} K(x-y)dx -\int_{W_Q\backslash D_Q} \partial_x^{n}  K(x-y)dx +$$
$$+ PV\int_{W_Q}  \partial_x^{n} K(x-y)dx=IV_1+IV_2+V =IV+V.$$
The domain $S=D_Q \backslash W_Q \cup W_Q\backslash D_Q$  and the boundary $\partial D$ are tangent in origin;  hence, $S$ is small. An absolute value estimate of IV follows \cite{V1} and is given in the next section.

 The estimates of  the principle value  term $V$ in a polynomial domain  $W_Q$ are postponed to Sections 3.4-3.6.

\subsection{Estimate of integral over a small sector (term IV)}

Observe that the domain $ S$ is contained in the domain $\{x=(x',x_d):|x'|<r_0, \:|x_d|\lesssim \omega(|x'|)\}$. Therefore, we get
\[ |IV|\leq\int_{S}  |\partial_x^{n} K(x-y)dx|  \]
\[ =\int_{|x'|<r_0}dx' \left |\int_{P(x')}^{A(x')}\frac{dx_d}{(|x'|^2+(x_d-\delta)^2)^{(d+n)/2}}\right| \]
\[ =\int_{|x'|<r_0}dx'\int_{|x_d|\leq |A(x')-P(x')|}\frac{dx_d}{(|x'|^2+(x_d+ P(x')-\delta)^2)^{(d+n)/2}} \]
\[ \leq \int_{|x'|<r_0}dx' \int_{|x_d|\lesssim \omega(|x'|)}\frac{dx_d}{(|x'|^2+(x_d+P(x')-\delta)^2)^{(d+n)/2}}.\]
Split the outer integral at the level $\delta$:
\[\int_{|x'|<r_0}=\int_{|x'|<\delta}+\int_{\delta<|x'|<r_0}.\]
For the first integral above, observe that   the estimates $|x_d|\lesssim \omega(|x'|)=o(|x'|)$  and $|P(x')|\lesssim |x'|^2$ imply that \[|x'|^2+(x_d+P(x')-\delta)^2\gtrsim \delta^2\]
for $\delta$ small enough.
 Hence, we have
 \[\int_{|x'|<\delta}dx'\int_{|x_d|\lesssim \omega(|x'|)}\frac{dx_d}{(|x'|^2+(x_d+P(x')-\delta)^2)^{(d+n)/2}}\]
 \[\lesssim \frac{1}{\delta^{d+n}}\int_{|x'|<\delta} dx'\int_{|x_d|\lesssim \omega(|x'|)}dx_d\]
 \[\lesssim \frac{1}{\delta^{d+n}}\int_{|x'|<\delta}\omega(|x'|)dx'\]
 \[\lesssim \frac{\omega(\delta)}{\delta^{n+1}}.\]
For the second integral,  we  use the estimate $|x'|^2+(x_d+P(x')-\delta)^2\geq |x'|^2$, and
  by almost decreasing property (\ref{eq:eq42}), we obtain
\[\int_{\delta<|x'|<r_0}dx'\int_{|x_d|\lesssim \omega(|x'|)}\frac{dx_d}{(|x'|^2+(x_d+P(x')-\delta)^2)^{(d+n)/2}}\]
\[\lesssim \int_{\delta< |x'|< r_0}\frac{dx'}{|x'|^{d+n}}\int_{|x_d|\lesssim \omega(|x'|)}dx_d\]
\[ \lesssim\int_{\delta< |x'|< r_0}\frac{\omega(|x'|)dx'}{|x'|^{d+n}}\]
\[\lesssim \frac{\omega(\delta)}{\delta^{n+1}}.\]
 So, integral $IV$ is estimated.

\subsection{Estimate of integral over a polynomial domain (term V)}

Consider the upper half-space
$\Pi_+ =\{x=(x',x_d) \in  \mathbb{R}^d:\: x_d>0$\}.  By Lemma \ref{lem15}, we have
\[\int_{W_Q}  \partial_x^{n} K(x-y)dx\]
\[= \int_{W_Q}  \partial_x^{n} K(x-y)dx-\int_{\Pi_+}  \partial_x^{n} K(x-y)dx\]

\[= \int_{W_Q}  \partial_x^{n} K(x-y)dx-\int_{\Pi_+\cap Q}  \partial_x^{n} K(x-y)dx+\int_{\Pi_+\setminus Q}  \partial_x^{n} K(x-y)dx.\]
 For the third integral above  we have
\[\int_{\Pi_+\setminus Q} | \partial_x^{n} K(x-y)|dx\]
\[\lesssim\int_{|x|>r_0}  |\partial_x^{n} K(x-y)|dx\lesssim\frac{1}{r_0^{n}} .\]
It remains to estimate the following quantity
\[J= \int_{W_Q}  \partial_x^{n} K(x-y)dx-\int_{\Pi_+\cap Q}  \partial_x^{n} K(x-y)dx\]
\[=-\int_{ |x'|< r_0}dx'\int_0^{P(x')}\partial_x^{n} K(x-y)dx_d.\]
The domain of integration is tangent to the boundary in origin, but it is not small enough. Indeed, the inequality $0<x_d< P(x')$  implies only that $|x_d|<|x'|^2$, and we can not follow the   proof of  term IV in previous Section 3.3, where the estimate $|x_d|<\omega(|x'|)$ is crucial. Our approach  is to  expand the integral $J$ into a sum of terms of two types.  Using polynomial restrictions on the integration domain we will evaluate the first ones  explicitly, while the latter are evaluated absolutely.

For this consider the  Taylor expansion of  the inner integral up to the order $n$ with respect to the variable $x_d$ around $0$. We have
 \[\int_0^{P(x')}\partial_x^{n} K(x',x_d-\delta)dx_d\]
\[=\sum_{k=1}^{n-1}\frac{1}{k!}\partial_d^{k-1}\partial_x^{n} K(x',-\delta)P(x')^k+\]
\[ + \frac{1}{(n)!}\partial_d^{n}\partial_x^{n} K(x',\eta-\delta)P(x')^{n}\]
\[=\sum _{k=1}^{n}I_k,\]
where $\eta \in (0, P(x'))$ is an appropriate point and $\partial_d$ is the derivative with respect to $x_d$.

We represent each polynomial $P(x')^k$  as  the  sum:
\begin{equation}
\label{eq:eq50}
  P(x')^k= Q_{n+k}(x')+T_{n+k}(x'), k=1,\dots,n
\end{equation}
where $Q_{n+k}$ is a polynomial of order less than  $n+k$ and $T_{n+k}$ has no terms of order less then $ n+k$.

Since the polynomial $P(x')$ has no terms of order less than 2, $P(x')^k$ and hence $Q_{n+k}$  have no terms of order less than $2k$. In particular,  $Q_{2n}$ has no terms of order less than $2n$, hence $Q_{2n}$ is zero.
 So, we get
\[I_k= \partial_d^{k-1}\partial_x^{n} K(x',-\delta)Q_{n+k}(x')+
\partial_d^{k-1}\partial_x^{n} K(x',-\delta)T_{n+k}(x')\]
 \[=I'_k+I''_k,\]
for $k=1,\dots , n-1$ and respectively
\[I_{n}=\partial_d^{n-1}\partial^{n} K(x',\eta -\delta)T_{2n}(x')=I''_{n}.\]

\subsection{Estimates of integrals with $I'_k$-terms }

 To estimate the integrals with $I'_k$-terms, we write $Q_{n+k}=\sum_{i=2k}^{n+k-1}q_i$,  $k=1,\dots ,n-1$, where $q_i$ is  a   homogeneous polynomials   of order $i$. In what follows, for simplicity, we  write  $\partial^{n+k-1}$ instead of $\partial_{d}^{k-1}\partial_x^{n}$. Then, one has
\[\int_{ |x'|< r_0} I'_k(x') dx'\]
\[=\int_{ |x'|< r_0}\partial^{n+k-1} K(x',-\delta)Q_{n+k}(x')dx'\]
\[=\sum_{i=2k}^{n+k}\int_{ |x'|< r_0}\partial^{n+k-1} K(x',-\delta)q_i(x')dx'.\]
Consider the summand
\[J_i=\int_{ |x'|< r_0}\partial^{n+k-1} K(x',-\delta)q_i(x')dx'\]
\[=\int_{ \mathbb{R}^{d-1}}\partial^{n+k-1} K(x',-\delta)q_i(x')dx'-\int_{ |x'|> r_0}\partial^{n+k-1} K(x',-\delta)q_i(x')dx'.\]
The second term  is   bounded by a constant $C(r_0)$. Indeed, we have
$$\int_{ |x'|> r_0}|\partial^{n+k-1} K(x',-\delta)q_{i}(x')|dx'$$
$$\lesssim\int_{ |x'|> r_0}|\partial^{n+k-1} K(x',-\delta)||x'|^{i}dx'$$
$$\lesssim\int_{ |x'|> r_0}|x'|^{-n-k+1-d+i}dx'$$
$$\lesssim\int_{ r_0}^\infty \frac{dr}{r^{n+k+1-i}}$$
$$\lesssim\frac{{r_0}^{i}}{{r_0}^{n+k}}<\infty.$$

    For the first term, we have
$$\int_{ \mathbb{R}^{d-1}}\partial^{n+k-1} K(x',-\delta)q_i(x')dx'=$$
$$\delta^{-n-k+1+i}\int_{ \mathbb{R}^{d-1}}\partial^{n+k-1} K(x',-\delta)q_i(x')\delta^{n+k-1-i}dx'.$$
By Lemma \ref{lem17}, the function $N(x',\delta)=\partial^{n+k-1} K(x',-\delta)q_i(x')\delta^{n+k-1-i}$ is an even   smooth homogeneous Calder\'{o}n-Zygmund kernel. By homogeneity and by Lemma \ref{lem16}, we have
 $$\int_{ \mathbb{R}^{d-1}}\partial^{n+k-1} K(x',-\delta)q_i(x')dx'=0$$
and  hence, $|J_i|\leq C(r_0)$ for every $i=2k,\dots , n+k-1$. This clearly implies
 \[\left|\int_{ |x'|< r_0} I'_k(x') dx'\right|<C(r_0)<\infty\]
for $k=1,\dots,n-1$. Recall that $I'_{n}=0$; thus, all terms $I'_k$ are uniformly bounded by a constant independent of $\delta$.

\subsection{Estimates of integrals with $I''_k$-terms }
The integrals with terms $I''_k$ will be   estimated absolutely. Recall that $T_{k}$ does not contain terms of order less than $\leq n+k$; hence, we get
\[|T_{k}(x')|\lesssim |x'|^{n+k}\]
for  $|x'|<r_0$ and $k=1,\dots, n$.  Also, observe that
\[|\partial^{n+k-1} K(x',-\delta)|\lesssim\frac{1}{(|x'|^2+\delta^2)^{(n+k-1+d)/2}}\]
for $k=1,\dots, n-1$  uniformly with respect to  $x\in D$.

  For $k=n$,  observe that $\eta \in (0,P(x'))$ implies $|\eta|=o(|x'|)$;
  hence $|x'|^2+(\eta-\delta)^2\gtrsim |x'|^2+\delta^2$. So, we obtain
\[|\partial^{2n-1} K(x',\eta -\delta)|\lesssim\frac{1}{(|x'|^2+\delta^2)^{(2n-1+d)/2}}\]
 uniformly with respect to $x\in \partial D$. Therefore, for $k=1,\dots, n$,
$$\int_{ |x'|< r_0}|I''_k(x')|dx'\lesssim\int_{ |x'|< r_0}\frac{|x'|^{n+k}dx'}{(|x'|^2+\delta^2)^{(n+k-1+d)/2}}.$$
Next, split the integral under consideration at the level $\delta$ as
   $$ \int_{|x'|<r_0} =\int_{|x'|<\delta}+\int_{\delta< |x'|< r_0}.$$
For the first integral on the right hand side,  we have $|x'|^2+\delta^2\geq \delta^2$; hence,
 $$\int_{|x'|<\delta} \frac{|x'|^{n+k}dx'}{(|x'|^2+\delta^2)^{(n+k-1+d)/2}}\leq \int_{|x'|<\delta}\frac{|x'|^{n+k}dx'}{\delta^{n+k-1+d}}\lesssim C$$
with $C$ independent of $\delta$.

For the second integral on the right hand side, we have  $|x'|^2+\delta^2\geq|x'|^2$; hence,
  $$\int_{\delta< |x'|< r_0}\frac{|x'|^{n+k}dx'}{(|x'|^2+\delta^2)^{(n+k-1+d)/2}}\lesssim \int_{\delta< |x'|< r_0}\frac{|x'|^{n+k}dx'}{|x'|^{d+n+k-1}}$$
  $$\lesssim \log\frac{1}{\delta}.$$
  Therefore,
  $$\int_{ |x'|< r_0}|I''_k(x')|dx'\lesssim 1+ \log\frac{1}{\delta}\lesssim \frac{\om(\delta)}{\delta^{n+1}}$$
 by (\ref{eq:eq42}).

With estimates of terms IV and V in hand,  the proof of Theorem \ref{thm5} is completed.

\section{Proof of Theorem \ref{thm4}}

\subsection {Gradient estimates for polynomials.}
To check items (1) and (2)   of Theorem \ref{thm3} for polynomials, we need  the    gradient estimates similar to Theorem \ref{thm5}. Based on the estimate (\ref{eq:eq222}) these ones are  easier.
\begin{lem}
\label{lem 10}
    Let $\om(t)$ be a growth function of  type $n$, $n\geq 2$, let $D\subset \mathbb{R}^d$ be a   $C^{\om}$ domain and let $T$ be  a $C^{2n-1}$-smooth  homogeneous Calder\'{o}n-Zygmund operator with an even kernel.   Let $q$, $n<q<n+1$, be a parameter of the almost decreasing property (\ref{eq:eq42}) from Definition \ref{df1}. Then for any  homogeneous polynomial $S$ of degree $k$, $k=1,\dots, n-1$ we have uniformly in $x$
  \begin{equation}
 \label{eq:eq223}
   | \nabla ^{n}T_D S_x(x)|\lesssim \rho(x)^{q-n-1},
    \end{equation}
where $S_x(\cdot)=S(\cdot-x)$.
\end{lem}
\begin{proof}
We will prove the esimate (\ref{eq:eq223}) for  each derivative $\partial_x^{n}$ instead of $\nabla ^{n}$. One gets
$$\partial_x^{n} \int_D K(y-x)S(y-x)dy =\int_D  \partial_x^{n}(K(y-x)S(y-x))dy $$
$$=(-1)^k \partial_x^{n-k}\int_D  \partial_y^{k}(K(y-x)S(y-x))dy. $$
By Lemma \ref{lem17}, $N(x)= \partial_x^{k}(K(x)S(x))$,  $k=1,\dots, n-1$, is  a    $C^{2n-k-1}$-smooth even homogeneous Calder\'{o}n-Zygmund kernel.
 We check the assumptions  of   Theorem \ref{thm5} by redesignation of parameters.
  Denote $l=n-k$.  Since $2n-k-1 \geq 2l-1$, we have that $N$ is  a   $C^{2l-1}$-smooth even homogeneous Calder\'{o}n-Zygmund kernel.
 Consider the boundary parametrization from Section 3.1: there are  a function $A$ with   a Taylor polynomial $P_{n}(x')$ around origin  and a constant $c_0$ such that
   \[ |A(x')-P_{n}(x')|\leq c_0 \omega(|x'|), \:|x'|<r_0 .\]
   For     the Taylor polynomial $P_{l}$ of order $l$ of the function $A$ around origin we have
  \[I= |A(x')-P_{l}(x')|\leq c_1 \max (|x'|^{l+1}, \omega(|x'|)), \]
 where $ |x'|<r_0$, $l=1,\dots, n-1$ and $c_1=c_1(D)$. By (\ref{eq:eq42}), choose $q$, $n<q<n+1$, such  that
 $\frac{\omega(t)}{t^{q}}$ is almost decreasing. Particularly, we  have
  $t^{q}\lesssim \om(t)\lesssim t^{q-1}$.
 Then
  \[I\lesssim \max (|x'|^{l+1}, |x'|^{q-1})\lesssim |x'|^{l+q-n},\]
 for $ l=1,\dots, n-1$.

Theorem \ref{thm5} is applicable with the following parameters:   the  growth function $\om'(t)=t^{l+q-n}$ of type $l$,  the  bounded domain  $D\subset \mathbb{R}^d$  with $C^{\om'}$  boundary,    and   the $C^{2l-1}$-smooth  even  Calder\'{o}n-Zygmund kernel $N$,  So, we obtain
\[\left|\partial_x^{n} \int_D K(y-x)S(y-x)dy\right|= \left|\partial_x^{l} \int_D N(y-x)dy\right|\]
\[\lesssim \frac{\om'(\rho(x))}{\rho(xy)^{l+1}}=\rho(x)^{q-n-1}\]
if $l=2,\dots, n-1$.

Observe, that formally, we can not apply Theorem \ref{thm5} for $l=1$. In this case, it holds by Proposition
2  \cite{V1}.
\end{proof}
\begin{lem}
\label{lem 20}
Let $\om(t)$ be a growth function of  type $n$, $n\geq 2$, let $D\subset \mathbb{R}^d$ be a   $C^{\om}$ domain and let $T$ be  a $C^{2n-1}$-smooth  homogeneous Calder\'{o}n-Zygmund operator with an even kernel.   Let $q$, $n<q<n+1$,   be a parameter of the almost decreasing property (\ref{eq:eq42}) from Definition \ref{df1}. Then  for any  homogeneous polynomial $S$ of degree $n-1$, and for any cube $Q$, $2Q\in D$, with the centre $x_0$ we have
\begin{equation}\label{eq:eq64}
|\nabla^{n} T_D S_{x_0}(x)|\lesssim\sum_{i=0}^{n-1} |x-x_0|^i \rho(x_0)^{q-n-1} +|x-x_0|^n  \frac{\om(\rho(x_0))}{\rho(x_0)^{n+1}}.
\end{equation}
with a constant independent of $Q$.
\end{lem}
\begin{proof}

We will prove the esimate above for each derivative $\partial_x^{n}$ instead of $\nabla ^{n}$.
Consider a homogeneous expansion into the finite sum with respect to variable $y$ around $x\in Q$
\[S_{x_0}(y)=S(y-x+x-x_0)=\sum S'(x-x_0)S''(y-x),\]
where $S'$ and $S''$ are certain homogeneous polynomials in  $d$ real variables, and  $deg S'=k$, $deg S''=n-1-k$, $k=0,\dots, n-1$.
Differentiating by the Leibniz rule, we represent $\partial_x^{n} T_D S_{x_0}(x)$ as a finite sum of next summands
\[\partial_x^{i}S'(x-x_0)\partial_x^{n-i} \int_D K(x-y)S''(y-x)dy \]
\[=\partial_x^{i}S'(x-x_0)  \partial_x^{1+k-i} (-1)^{n-1-k}\int_D  \partial_y^{n-1-k}(K(x-y)S''(y-x))dy \]
with certain derivatives $\partial^i_x$  of order $i=0,\dots, k$.
Clearly, we have
\begin{equation}\label{eq:eq69}
|\partial_x^{i}S'(x-x_0)|\lesssim |x-x_0|^{k-i}.
 \end{equation}

 To estimate  a quantity
\[ | \partial_x^{1+k-i} \int_D  \partial_y^{n-1-k}(K(x-y)S''(y-x))dy | \]
 we follow Lemma \ref{lem 10}. So,
 denote $l=1+k-i$, $l=1,\dots, n$, define  the kernel $N=\partial^k(KS'')$, and take  a type $l$ growth function $\om_l(t)=t^{l+q-n}$.
 We obtain
\[|\partial_x^{l} \int_D N(x-y)dy|\lesssim \frac{\rho(x)^{l+q-n}}{\rho(x)^{l+1}}\lesssim \rho(x)^{q-n-1},\]
 when $l=1,\dots, n-1$.

If $l=n$, then $i=k=0$; so exactly by Theorem \ref{thm5}  one gets
\[|\partial_x^{n} \int_D  K(x-y)dy|\lesssim \frac{\om(\rho(x))}{\rho(x)^{n+1}}.\]

Combining this with the estimate (\ref{eq:eq69}) one has
\[|\partial_x^{n} T_D S_{x_0}(x)|\lesssim\sum_{i=0}^{n-1} |x-x_0|^i \rho(x)^{q-n-1} +|x-x_0|^n  \frac{\om(\rho(x))}{\rho(x)^{n+1}},\]
and we obtain (\ref{eq:eq64}) for $x\in Q$.
\end{proof}

 We need two lemmas, which are valid in the general Lipschitz domains.  Define the function that is the the denominator of $\omt$ in (\ref{eq:eq2}):
 \begin{equation}\label{eq:eq67}
\xi(x)=\max\{1, \int_{x}^{1}\om (t) t^{-n-1}dt\}.
 \end{equation}

\begin{lem}[{\cite[Lemma 2.8]{VD}}]
  \label{lem46}
 Given a growth function  $\om$  of type $n$,  a  bounded Lipschitz domain  $D\subset \mathbb{R}^d$, a function $f\in C_*^{\om}(D)$, a cube $Q$ with size length $\ell$ and a polynomial $P_{f,Q}\in \mathcal{P}_n$ such that    $\|f-P_{f,Q}\|_{L^{\infty}(Q)}\lesssim\|f\|_{\omega,D}\om(\ell).$
 Then the estimate
 \[\|P_{f,Q}\|_{L^{\infty}(D)}\lesssim \|f\|_{\om,D}\xi(\ell)\]
 holds with a constant  independent of $f$.
\end{lem}
 The second one is a variant of Marchaud argument (see for powered growth functions, e.g., \cite[Proposition 4.5]{KK}).
\begin{lem}
  \label{lem44}
 Given  growth function $\om$  of type $n$,  a  bounded domain  $D\subset \mathbb{R}^d$  with  the Lipschitz   boundary, a function $f\in C_*^{\om}(D)$ and a polynomial $S$,   the estimate
 \begin{equation}\label{eq:eq72}
 \|fS\|_{\om,D}\lesssim \|S\|_{L^{\infty}(D)}\|f\|_{\om,D}
 \end{equation}
  holds with a constant  independent of $f$.
\end{lem}
\begin{proof}
 Let $Q$ be a cube such that $2Q\subset D$.
 Since $f\in C_\om(D)$, by Definition \ref{df2} one has   $\|f-P_{f,Q}\|_{L^{\infty}(Q)}\lesssim \|f\|_{\omega,D}\om(\ell)$,
 with a certain polynomial $P_{f,Q}\in \mathcal{P}_n$.  Hence,
 \begin{equation}\label{eq:eq62}
 \|Sf-SP_{f,Q}\|_{L^{\infty}(Q)}\lesssim \|S\|_{L^{\infty}(D)}\|f\|_{\omega,D}\om(\ell).
 \end{equation}
The conclusion of the lemma  holds if we may replace the polynomial $SP_{f,Q}$ by a polynomial of degree $n$ in inequality (\ref{eq:eq62}). To do this, take for  $S P_{f,Q}$ the Taylor polynomial $T_{x_0}$ of degree $n$ around the centre $x_0$. We have
\[I=\|SP_{f,Q}-T_{x_0}\|_{L^{\infty}(Q)}\lesssim_{n,d}\|\nabla^{n+1} (SP_{f,Q})\|_{L^{\infty}(Q)} \ell^{n+1}\]
\[\lesssim_{n,d}\|\nabla^{n+1} (SP_{f,Q})\|_{L^{\infty}(D)} \ell^{n+1}\]
\[\lesssim_{n,d}\|S\|_{L^{\infty}(D)}\|P_{f,Q}\|_{L^{\infty}(D)} \ell^{n+1}\]
\[\lesssim_{n,d}\|S\|_{L^{\infty}(D)} \|f\|_{\om,D} \xi(\ell) \ell^{n+1},\]
by Bernstein inequality on $D$ and by Lemma \ref{lem46}.
Since almost decreasing property (\ref{eq:eq42}), we have       $\xi(\ell)\ell^{n+1}\lesssim \om(\ell)$, and therefore,
\[I\lesssim \|S\|_{L^{\infty}(D)}\|f\|_{\om,D} \om(\ell),\]
as required.
\end{proof}
\subsection{Checking of point (1) of Theorem \ref{thm3} for  constants.}
\begin{proof}
By   Theorem \ref{thm5} applied for the following parameters: the type $n+1$, $n\geq 1$, growth function $t \om(t)$, $C^{\om,1}$ domain $D$ and $C^{2n+1}$ -smooth  homogeneous Calder\'{o}n-Zygmund operator $T_D$ with an even kernel, one has the gradient estimate
\[|\nabla^{n+1}T_D\chi_D(x)|\lesssim \frac{\rho(x)\om(\rho(x))}{\rho^{n+2}(x)}=\frac{\om(\rho(x))}{\rho^{n+1}(x)},\;x\in D.\]
    By Taylor's formula  for  $T_D\chi_D$  with the Taylor polynomial $P_{n,x_0}$ of  degree $n$ around the centre $x_0$ of a cube $Q$ such that   $2Q\subset D$, we get
\[
\begin{aligned}
  |T_D\chi_D(x)-P_{n,x_0}(x)|  & \lesssim \sup_Q|\nabla^{n+1}T_D\chi_D(x)||x-x_0|^{n+1}  \\
  & \lesssim \frac{\om(\rho(x))}{\rho^{n+1}(x)}|x-x_0|^{n+1}\lesssim \om(|x-x_0|)
\end{aligned}
\]
uniformly with respect on $ x\in Q$.  By Proposition 2.2 \cite{VD},
 $T_D\chi_D \in C^{\om}(D)$, which clearly follows $T_D\chi_D \in C^{\om}_*(D)$, and  item (1) of Theorem \ref{thm3} for constants is completed.
\end{proof}
\subsection { Checking of point (1) of Theorem \ref{thm3} for polynomials.}
\begin{proof}
We repeat the  estimates  of the reminder in the Taylor's formula. For this put the type $n+1$ growth function $t \om(t)$, $C^{\om,1}$ domain $D$ and $C^{2n+1}$ -smooth
homogeneous Calder\'{o}n-Zygmund operator $T_D$ with an even kernel. Take a polynomial
    $S\in\mathcal{P}_n$  and  a cube $Q$ such that $2Q\subset D$.
We have
\[T_DS=\int_D K(x-y)S(x)dx\]
\[=\int_D K(x-y)(S(x)-S(y))dx+S(y)\int_D K(x-y)dx\]
\[=\int_D K(x-y)(S(x)-S(y))dx+S(y)T_D\chi_D(y)=A+B.\]
 For term A we will apply  the Taylor's formula. So, consider the homogeneous expansion
\[S(x)-S(y)=\sum_{k=1}^{n}S_k(y-x)\]
where $S_k$ are homogeneous polynomials of degree $k,\; k=1,\dots, n$. Define $S_{k,x}(\cdot)=S_k(\cdot-x)$ and put $Q$ be a cube with centre $x_0$  and $2Q\subset D$.    For  each $T_DS_{k,x}\chi_D$   the  relevant reminder in the Taylor's formula of $n$ order  by Lemma \ref{lem 10} for the growth function  $t\om(t)$ of type $n+1$  is estimated  as following
   \[ \sup_Q|\nabla^{n+1}T_DS_{k,x}(x)\chi_D(x)||x-x_0|^{n+1}\]
  \[ \lesssim \rho(x_0)^{q+1-n-2}|x-x_0|^{n+1}\]
     \[ \lesssim |x-x_0|^{q}\lesssim \om(\ell) .\]
For  term B we apply Lemma \ref{lem44} with $f=T_D\chi_D $ and easily obtain (\ref{eq:eq72}).  This completes point (1) of Theorem \ref{thm3}.
\end{proof}

\subsection{Proof of  point (2) of Theorem \ref{thm3}. }
To finish the proof of Theorem \ref{thm4} it remains to check point (2) of Theorem \ref{thm3}.
 Take a cube $Q$, $2Q\in D$, with the centre $x_0$ and   a  homogeneous  polynomial $P_{x_0}(x)= P(x-x_0)$ of order $n>0$. We again estimate the reminder of  the Taylor's formula for  $T_D\chi_D$  and its Taylor polynomial $P_{n,x_0}$ of  degree $n$ around the centre $x_0$ of a cube $Q$ with size length $\ell$.

For any parameter $q$ such that   the almost decreasing condition (\ref{eq:eq42}) holds, we may replace it by any other $q'$,  $q<q'<n+1$, so  that  $\frac{\omega(t)}{t^{q'}}$ is almost decreasing as well, and therefore, $\om(t)\gtrsim t^{q'}$. Given $q$, we choose this $q'$ in the form
$q'=q+n-r$, where $r$ is a parameter from the almost increasing condition (\ref{eq:eq32}) such that $r+1>q$.  Then (\ref{eq:eq64}) is replaced by
 \begin{equation}\label{eq:eq63}
|\partial_x^{n+1} T_D P_{x_0}(x)|\lesssim\rho(x_0)^{q'-n-1}+   \frac{\om(\rho(x_0))}{\rho(x_0)}.
 \end{equation}
We have a simple  estimate on  $\omt(t)$ from below.
  \begin{lem}\label{lem47}
   $\omt(t)\gtrsim t^{q'}$.
  \end{lem}
  \begin{proof}
   Consider the function $\xi(x)$ from (\ref{eq:eq67}). By almost increasing property (\ref{eq:eq32}), one has
  \[\xi(x) \approx \int_x^1 \omega(t)t^{-n-1}dt=\int_x^1\frac{\omega(t)}{t^{r}} t^{r-n-1}dt \lesssim x^{r-n}.\]
  Therefore, by almost decreasing  property (\ref{eq:eq42}),
  \[\omt(x)=\frac{\om(x)}{\xi(x)}\gtrsim \om(x) x^{n-r}\gtrsim x^{q+n-r},\]
  that is required.
  \end{proof}

By Lemma \ref{lem47}, for the first summand in (\ref{eq:eq63}) we have
 \[ \rho(x_0)^{q'-n-1}\lesssim \frac{\omt(\rho(x_0))}{\rho(x_0)^{n+1}}\]
 and for the second one it simply holds
\[  \frac{\om(\rho(x_0))}{\rho(x_0)}\lesssim  \frac{\omt(\rho(x_0))}{\rho(x_0)^{n-r+1}} \]
\[ \lesssim   \frac{\omt(\rho(x_0))}{\rho(x_0)^{n+1}}.\]
Combining, by the Taylor's formula we obtain  point (2) of  Theorem \ref{thm3}. The proof of Theorem \ref{thm4} is completed.





\bibliographystyle{elsarticle-num}

\begin{thebibliography}{00}


\bibitem{An}
D.S. Anikonov, \textit{On the boundedness of a singular integral operator in the space $ C^{\alpha}(\overline G),$}
 Math. USSR-Sb. 33(4)(1977), 447-464.
\bibitem{CT}
V. Cruz, X. Tolsa, \textit {Smoothness of the Beurling transform in Lipschitz domains}. J. Funct. Anal. 262(10)(2012), 4423-4457.

\bibitem{I}
 T. Iwaniec, \textit{The best constant in a BMO-inequality for the Beurling-
Ahlfors transform}, Mich. Math. J. 34 (1987), 407-434.
\bibitem{Ja2}
S. Janson, \textit{Generalizations of Lipschitz spaces and an application to Hardy spaces and bounded mean oscillation.}
 Duke Math. J. 47(1980) 959-982.


\bibitem{KK}
 S. Kislyakov, N. Kruglyak, \textit{Extremal problems in interpolation theory, Whitney--Besicovitch
coverings, and singular integrals.}  Birkhauser/Springer, Basel AG, 2013.

\bibitem{MOV}
J. Mateu, J. Orobitg and J. Verdera, \textit{Extra cancellation of even Calder\'{o}n-Zygmund operators
and quasiconformal mappings,} J. Math. Pures Appl. (9) 91 (2009), no. 4, 402-431.

\bibitem{Mikh}
S.G. Mikhlin, \textit{Multidimensional singular integrals and integral equations.}   Pergamon Press, Oxford,  1965


\bibitem{Pr}
M. Prats, \textit{	Sobolev regularity of the Beurling transform on planar domains} Publ. Mat. 61 (2017)  291-336	

\bibitem{PT}
M. Prats, X. Tolsa, \textit{A T(P) theorem for Sobolev spaces on domains,} J. Funct. Anal.
268 (2015) 2946-2989.
\bibitem{T}
X. Tolsa,  \textit{Regularity of C1 and Lipschitz domains in terms of the Beurling transform}. J. Math. Pures Appl. 100.2 (2013): 137-165.



\bibitem{V1}
A.V.Vasin,   \textit{A T1 theorem and Calder\'{o}n–Zygmund operators in Campanato spaces on domains.}
Math. Nachr.  292(2019), 1392-1407.



\bibitem{VD}

A.V.Vasin,  E. Doubtsov, \textit{A T(P) theoremA for Zygmund spaces on domains.}
Math. Notes, 114 (2023),  1, 30-45.

\end{thebibliography}



\end{document}